\documentclass{article}
\usepackage[utf8]{inputenc}

\pdfoutput=1
\usepackage{amsmath}
\usepackage{amsfonts}
\usepackage{amssymb}
\usepackage{graphicx}
\usepackage{graphicx}%
\usepackage{tikz}
\usetikzlibrary{decorations.markings}
\usetikzlibrary{arrows}
\providecommand{\U}[1]{\protect\rule{.1in}{.1in}}
\newtheorem{theorem}{Theorem}

\newtheorem{corollary}[theorem]{Corollary}

\newtheorem{definition}[theorem]{Definition}

\newtheorem{lemma}[theorem]{Lemma}

\newenvironment{proof}[1][Proof]{\noindent\textbf{#1.} }{\ \rule{0.5em}{0.5em}}
\begin{document}

\title{Multivariate Alexander colorings}
\author{Lorenzo Traldi\\Lafayette College\\Easton Pennsylvania 18042, United States
}
\date{ }
\maketitle

\begin{abstract}
We extend the notion of link colorings with values in an Alexander quandle to link colorings with values in a module $M$ over the Laurent polynomial ring $\Lambda_{\mu}=\mathbb{Z}[t_1^{\pm1},\dots,t_{\mu}^{\pm1}]$. If $D$ is a diagram of a link $L$ with $\mu$ components, then the colorings of $D$ with values in $M$ form a $\Lambda_{\mu}$-module $\mathrm{Color}_A(D,M)$. Extending a result of Inoue [Kodai Math.\ J.\ 33 (2010), 116-122], we show that $\mathrm{Color}_A(D,M)$ is isomorphic to the module of $\Lambda_{\mu}$-linear maps from the Alexander module of $L$ to $M$. In particular, suppose $M$ is a field and $\varphi:\Lambda_{\mu} \to M$ is a homomorphism of rings with unity. Then $\varphi$ defines a $\Lambda_{\mu}$-module structure on $M$, which we denote $M_\varphi$. We show that the dimension of $\mathrm{Color}_A(D,M_\varphi)$ as a vector space over $M$ is determined by the images under $\varphi$ of the elementary ideals of $L$. This result applies in the special case of Fox tricolorings, which correspond to $M=GF(3)$ and $\varphi(t_i) \equiv-1$. Examples show that even in this special case, the higher Alexander polynomials do not suffice to determine $|\mathrm{Color}_A(D,M_\varphi)|$; this observation corrects erroneous statements of Inoue [J. Knot Theory Ramifications 10 (2001), 813-821; op. cit.].

\end{abstract}

\section{Introduction}

This paper is concerned with link invariants defined from diagrams. We use standard notation and terminology: A (tame, classical) \emph{link} $L=K_1 \cup \dots \cup K_{\mu}$ has $\mu$ disjoint components, each of which is a \emph{knot}, i.e., a piecewise smooth copy of $\mathbb{S}^1$ in $\mathbb{S}^3$. A \emph{diagram} $D$ of $L$ in the plane is obtained from a projection with only finitely many singularities, all of which are double points called \emph{crossings}. At each crossing, $D$ distinguishes the underpassing component by removing two short segments, one on each side of the crossing. Removing these segments splits $D$ into a finite number of arc components. The set of arc components of $D$ is denoted $A(D)$, and the set of crossings of $D$ is denoted $C(D)$. We also use standard notation for rings of Laurent polynomials with integer coefficients, $\Lambda=\mathbb{Z}[t^{\pm1}]$ and $\Lambda_{\mu}=\mathbb{Z}[t_1^{\pm1},\dots,t_{\mu}^{\pm1}]$.

The idea of a \emph{quandle} or \emph{distributive groupoid} was introduced in the 1980s by Joyce \cite{J} and Matveev \cite{M}. In the intervening decades a sizable literature has developed, involving many different generalizations and special cases of the quandle idea. In this paper we generalize one of these special cases. 

\begin{definition}
\label{aquandle}
An \emph{Alexander quandle} is a module $M$ over the ring $\Lambda$. The quandle operation is given by
\[
a_2 \triangleright a_1=(1-t) \cdot a_1 + t \cdot a_2.
\]
\end{definition}

Notice that for an Alexander quandle, the quandle operation is determined by the addition and scalar multiplication operations of the module. As we do not refer to any non-Alexander quandles in this paper, we use notation and terminology for modules rather than quandles. For instance, the following definition is equivalent to the definition of Alexander quandle colorings in the literature, even though the definition does not include the word ``quandle.''

\begin{definition}
\label{aquandlecolor}
Let $D$ be a link diagram, and $M$ a $\Lambda$-module. An \emph{Alexander coloring} of $D$ with values in $M$ is given by a function $f:A(D) \to M$ such that at every crossing $c$ as indicated in Figure~\ref{crossfig}, the following equation is satisfied:
\[
f(a_3)=(1-t) \cdot f(a_1) + t \cdot f(a_2).
\]
\end{definition}

\begin{figure}
\centering
\begin{tikzpicture} [>=angle 90]
\draw [thick] [->] (1,1) -- (-1,-1);
\draw [thick] (-1,1) -- (-.2,0.2);
\draw [thick] (0.2,-0.2) -- (1,-1);
\node at (1.3,1.3) {$a_1$};
\node at (-1.3,1.3) {$a_2$};
\node at (1.3,-1.3) {$a_3$};
\end{tikzpicture}
\caption{The arcs incident at a crossing.}
\label{crossfig}
\end{figure}
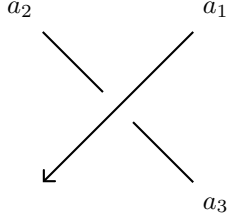

Here is a multivariate version of Definition~\ref{aquandlecolor}.

\begin{definition}
\label{aquandlecolormultiv}
Let $D$ be a diagram of a link $L=K_1 \cup \dots \cup K_{\mu}$, and let $M$ be a module over the ring $\Lambda_{\mu}$. Let $\kappa:A(D) \to \{1,\dots,\mu\}$ be the map with $\kappa(a)=i$ if and only if $a$ is an arc of $K_i$. Then a \emph{multivariate Alexander coloring} of $D$ with values in $M$ is given by a function $f:A(D) \to M$ such that at every crossing $c$ as indicated in Figure~\ref{crossfig}, the following equation is satisfied:
\[
f(a_3)=(1-t_{\kappa(a_2)}) \cdot f(a_1) + t_{\kappa(a_1)} \cdot f(a_2).
\]
The set of all multivariate Alexander colorings of $D$ with values in $M$ is denoted $\mathrm{Color}_A(D,M)$.
\end{definition}

Here are several remarks about these definitions.

1. Definition \ref{aquandlecolormultiv} includes Definition \ref{aquandlecolor}. If $M$ is a $\Lambda$-module then $M$ is also a $\Lambda_{\mu}$-module, with $t_i \cdot m = t \cdot m$ $\forall m \in M$ $\forall i \in \{1,\dots,\mu\}$. In particular, there is no difference between Definitions \ref{aquandlecolor} and \ref{aquandlecolormultiv} when $\mu=1$.

2. When we refer to Definition \ref{aquandlecolor} we sometimes use the phrase ``standard Alexander coloring'' to emphasize that we are not discussing Definition \ref{aquandlecolormultiv}.

3. Definition \ref{aquandlecolormultiv} does not seem to be associated with a notion of ``multivariate Alexander quandles'' analogous to the notion of standard Alexander quandles. There is no quandle structure on $M$ because $\kappa$ is defined on $A(D)$, not $M$.

4. Nosaka has pointed out that he mentioned the possibility of defining link colorings in $\Lambda_{\mu}$-modules in \cite[Remark 2.6]{N}. This idea was also mentioned by Manturov and Ilyutko \cite[Theorem 3.14]{Ma}, in the more general context of virtual links. These authors did not develop the results we present below, though.

5. For each $m \in M$, the constant function $f(a)=m$ satisfies Definition \ref{aquandlecolor} and the nonconstant function $f(a)=(1-t_{\kappa(a)}) \cdot m$ satisfies Definition \ref{aquandlecolormultiv}.

6. $\mathrm{Color}_A(D,M)$ is itself a module over $\Lambda_{\mu}$, using pointwise addition and scalar multiplication. That is, if $f_1,f_2 \in \mathrm{Color}_A(D,M)$ and $\lambda \in \Lambda_{\mu}$ then $(f_1+f_2)(a) = f_1(a)+f_2(a) \text{  and  } (\lambda \cdot f_1)(a) = \lambda \cdot f_1(a)$ $\forall a \in A(D)$.


Before stating results, we briefly recall some basic information about Alexander modules. We refer to the literature for more thorough discussions of these famous invariants of classical links \cite{BZ, CF, F, H}.

Each oriented link diagram $D$ has an associated \emph{Alexander matrix} $M(D)$. The columns of $M(D)$ are indexed by $A(D)$, and the rows of $M(D)$ are indexed by $C(D)$. Suppose $c$ is a crossing with the incident arcs indexed as in Figure~\ref{crossfig}. (N.b. The underpassing arcs are indexed using the orientation of $a_1$: $a_2$ is on the right side of an observer facing forward on $a_1$, and $a_3$ is on the left side.) If $a_2 \neq a_3$, then the row of $M(D)$ corresponding to $c$ has these entries:

\[
M(D)_{ca}=
\begin{cases}
1-t_{\kappa(a_2)}\text{,} & \text{if }a=a_1 \\
t_{\kappa(a_1)}\text{,} & \text{if }a=a_2 \\
-1\text{,} & \text{if }a=a_3 \\
0\text{,} & \text{if }a \notin\{a_1,a_2,a_3\}
\end{cases}
\]
If $a_2 = a_3$, then the row of $M(D)$ corresponding to $c$ has these entries:

\[
M(D)_{ca}=
\begin{cases}
1-t_{\kappa(a_2)}\text{,} & \text{if }a=a_1 \\
t_{\kappa(a_1)}-1\text{,} & \text{if }a=a_2=a_3 \\
0\text{,} & \text{if }a \notin\{a_1,a_2\}
\end{cases}
\]
The reader familiar with the free differential calculus will recognize that the entries of the $c$ row of $M(D)$ are the images in $\Lambda_{\mu}$ of the free derivatives of the Wirtinger relator $a_1a_2a_1^{-1}a_3^{-1}$ corresponding to the crossing $c$.

\begin{definition}
\label{almod}
If $D$ is a diagram of $L$ then the \emph{Alexander module} $M_A(L)$ is the $\Lambda_{\mu}$-module presented by $M(D)$.
\end{definition}

That is to say, if $D$ is a diagram of $L$ and $\Lambda_{\mu}^{A(D)}$ is the free $\Lambda_{\mu}$-module on the set $A(D)$, then $M_A(L)$ is isomorphic to the quotient of $\Lambda_{\mu}^{A(D)}$ by the submodule $S$ generated by all elements of the form
\[
(1-t_{\kappa(a_2)}) \cdot a_1 + t_{\kappa(a_1)} \cdot a_2 - a_3
\]
where the arcs $a_1,a_2,a_3$ appear at a crossing of $D$ as in Figure \ref{crossfig}.

If $M$ is a $\Lambda_{\mu}$-module and $f:A(D) \to M$ is an arbitrary function, then $f$ defines a $\Lambda_{\mu}$-linear map $\widehat{f}:\Lambda_{\mu}^{A(D)} \to M$. This map $\widehat{f}$ defines a $\Lambda_{\mu}$-linear map with domain $M_A(L)$ if and only if $S \subseteq \ker(\widehat{f})$. We deduce the following result, which we call the Fundamental Theorem of Alexander colorings.

\begin{theorem}
\label{main1}
Let $D$ be a diagram of $L=K_1 \cup \dots \cup K_{\mu}$, and let $M$ be a module over $\Lambda_{\mu}$. If $M_A(L)$ is the Alexander module of $L$, then
\[
\mathrm{Color}_A(D,M) \cong \mathrm{Hom}_{\Lambda_{\mu}}(M_A(L),M) .
\]
\end{theorem}

Many authors have discussed the fact that standard Alexander colorings are connected to the Alexander module, or to the Alexander polynomials \cite{B, EN, HHO, I, I2, J, KL, L, M, N}. In particular, Inoue \cite{I2} stated the following version of the fundamental theorem for standard Alexander colorings. Inoue's result involves the \emph{reduced} Alexander module $M_A^{red}(L)$, i.e., the $\Lambda$-module presented by a matrix obtained from an Alexander matrix $M(D)$ by replacing $t_1,\dots,t_{\mu}$ with a single variable, $t$. 

\begin{corollary} (\cite{I2})
\label{maincor1}
Let $D$ be a diagram of a link $L$, and let $M$ be a $\Lambda$-module. Then the $\Lambda$-module of Alexander colorings of $D$ with values in $M$ is isomorphic to $\mathrm{Hom}_{\Lambda}(M_A^{red}(L),M)$.
\end{corollary}

There is no easily computable set of complete invariants for modules over $\Lambda$ and $\Lambda_{\mu}$, so these modules can be difficult to work with. Theorem~\ref{main1} and Corollary~\ref{maincor1} yield more convenient results when $M$ is both a module over $\Lambda_{\mu}$ and a vector space over a field, because a vector space is characterized up to isomorphism by its dimension. Before stating results we recall a standard definition of classical knot theory.

\begin{definition}
Let $D$ be a diagram of a link $L=K_1 \cup \dots \cup K_{\mu}$. Then the \emph{elementary ideals} $E_j(L)$ are ideals of $\Lambda_{\mu}$, indexed by $j \geq 0 \in \mathbb{N}$.
\begin{itemize}
\item If $j \geq |A(D)|$, then $E_j(L)=\Lambda_{\mu}$. 
\item If $|A(D)|>j\geq |A(D)|-|C(D)|$, then $E_j(L)$ is the ideal of $\Lambda_{\mu}$ generated by the determinants of $(|A(D)|-j) \times (|A(D)|-j)$ submatrices of $M(D)$. 
\item If $j<|A(D)|-|C(D)|$, then $E_j(L)=(0)$.
\end{itemize}
\end{definition}

The elementary ideals of links have been studied thoroughly; see \cite{H} for a detailed account of the theory.

Suppose $F$ is a field, $\varphi:\Lambda_{\mu} \to F$ is a homomorphism of rings with unity, and $M$ is a vector space over $F$. Let $M_\varphi$ denote the $\Lambda_{\mu}$-module obtained from $M$ using $\varphi$. (That is, if $\lambda \in \Lambda_{\mu}$ and $m \in M$ then $\lambda \cdot m = \varphi(\lambda) \cdot m$.) In Section 2 we prove the following.

\begin{theorem}
\label{main2}
Let $F$ be a field, let $\varphi:\Lambda_{\mu} \to F$ be a homomorphism of rings with unity, and let $M$ be a vector space over $F$. Let $L$ be a $\mu$-component link, and let $j_0$ be the smallest index with $\varphi(E_{j_0}(L)) \neq 0$. Then for any diagram $D$ of $L$,
\[
\mathrm{Color}_A(D,M_\varphi) \cong \mathrm{Hom}_F(F^{j_0},M)_\varphi.
\]
It follows that $\mathrm{Color}_A(D,M_\varphi)$ is a vector space over $F$ of dimension $j_0 \cdot dim_F(M)$.
\end{theorem}

In case $M=F$, we have the following.

\begin{corollary}
\label{maincor2}
Let $F$ be a field, and let $\varphi:\Lambda_{\mu} \to F$ be a homomorphism of rings with unity. Let $L$ be a $\mu$-component link, and let $j_0$ be the smallest index with $\varphi(E_{j_0}(L)) \neq 0$. Then for any diagram $D$ of $L$,
\[
\mathrm{Color}_A(D,F_\varphi) \cong \mathrm{Hom}_F(F^{j_0},F)_\varphi.
\]
It follows that $\mathrm{Color}_A(D,F_\varphi)$ is a vector space over $F$ of dimension $j_0$.
\end{corollary}

If $\varphi(t_i)=\varphi(t_j)$ $\forall i,j$, then the colorings described in Corollary \ref{maincor2} are standard Alexander colorings. These colorings have been studied by Kauffman and Lopes \cite{KL}, who refer to them as colorings by linear Alexander quandles. The most familiar instances are the Fox colorings, which correspond to homomorphisms with $\varphi(t_i)=-1$ $\forall i$.

As far as we know, the precise statement of Corollary \ref{maincor2} has not appeared before, although a version of the special case for Fox colorings was announced recently \cite{STW}. Inoue \cite{I, I2} stated a similar result for standard Alexander colorings, with the elementary ideals replaced by the higher Alexander polynomials. In Section 3 we show that Inoue's version of Corollary \ref{maincor2} is incorrect even in the simplest case, i.e., Fox colorings of knots with $F=GF(3)$, the field of three elements.

After discussing examples in Sections 3 -- 5, we outline the extension of Theorem \ref{main2} from fields to principal ideal domains in Section 6.

\section{Proof of Theorem \ref{main2}}

Our proof of Theorem \ref{main2} begins with two lemmas, which provide useful properties of tensor products in conjunction with ring homomorphisms. Full accounts of the general theory of tensor products may be found in standard algebra texts, like \cite{La}. 

If $\varphi:\Lambda_{\mu} \to R$ is a homomorphism of commutative rings with unity and $M$ is an $R$-module, then we denote by $M_\varphi$ the $\Lambda_{\mu}$-module on $M$ with $\lambda \cdot m = \varphi(\lambda) \cdot m$ for $\lambda \in \Lambda_{\mu}$ and $m \in M$.


\begin{lemma} \label{adj} 
Let $\varphi:\Lambda_{\mu} \to R$ be a homomorphism of commutative rings with unity, and let $M$ be an $R$-module. If $D$ is a diagram of a $\mu$-component link $L$, then
\[
\mathrm{Color}_A(D,M_\varphi) \cong \mathrm{Hom}_{R}(R_\varphi \otimes_{\Lambda_{\mu}} \hspace{-0.25 em} M_A(L),M)_\varphi.
\]
\end{lemma}

\begin{proof}
The isomorphism 
\begin{equation}
\label{isom}
\mathrm{Hom}_{R}(R_\varphi \otimes_{\Lambda_{\mu}} \hspace{-0.25 em} M_A(L),M)_\varphi \cong \mathrm{Hom}_{\Lambda_{\mu}}(M_A(L),M_\varphi)
\end{equation}
is a special case of the general property that $\mathrm{Hom}$ and $\otimes$ define adjoint functors. This particular type of adjointness is mentioned (for instance) by Lang \cite[p. 637]{La}. The lemma follows from (\ref{isom}) and the fundamental theorem.
\end{proof}

\begin{lemma} \label{pres} Suppose $D$ is a diagram of a $\mu$-component link $L$ and $\varphi:\Lambda_{\mu} \to R$ is a homomorphism of commutative rings with unity. Then $\varphi(M(D))$ is a presentation matrix for the $R$-module $R_\varphi \otimes_{\Lambda_{\mu}} \hspace{-0.25 em} M_A(L)$.
\end{lemma}
\begin{proof} As $M(D)$ is a presentation matrix for $M_A(L)$, there is an exact sequence 
\[
\Lambda_{\mu}^{C(D)} \xrightarrow{f} \Lambda_{\mu}^{A(D)} \xrightarrow{g} M_A(L) \xrightarrow{} 0 \text{,}
\]
where $f$ is the homomorphism represented by the matrix $M(D)$. A standard property of tensor products is the fact that for any set $S$,
\[
R_\varphi \otimes_{\Lambda_{\mu}} \hspace{-0.25 em} \Lambda_{\mu}^{S} \cong R^S \text{}
\]
with $1 \otimes s$ corresponding to $s$ for each $s \in S$. Moreover, if $\mathrm{id}$ is the identity map of $R$ then the homomorphism
\[
\mathrm{id} \otimes f: R_\varphi \otimes_{\Lambda_{\mu}} \hspace{-0.25 em} \Lambda_{\mu}^{C(D)} \to R_\varphi \otimes_{\Lambda_{\mu}} \hspace{-0.25 em} \Lambda_{\mu}^{A(D)}
\] 
is represented by the matrix $\varphi(M(D))$, with respect to the natural bases.

Another standard property of tensor products is right exactness. This property guarantees that
\[
R_\varphi \otimes_{\Lambda_{\mu}} \hspace{-0.25 em} \Lambda_{\mu}^{C(D)} \xrightarrow{\mathrm{id}\otimes f} R_\varphi \otimes_{\Lambda_{\mu}} \hspace{-0.25 em} \Lambda_{\mu}^{A(D)} \xrightarrow{\mathrm{id}\otimes g} R_\varphi \otimes_{\Lambda_{\mu}} \hspace{-0.25 em} M_A(L) \xrightarrow{} 0
\]
is an exact sequence of $R$-modules. It follows that $\varphi(M(D))$ is a presentation matrix for $R_\varphi \otimes_{\Lambda_{\mu}} \hspace{-0.25 em} M_A(L)$. 
\end{proof}

\begin{corollary}
\label{lastcor}
Suppose $L$ is a $\mu$-component link, $F$ is a field and $\varphi:\Lambda_{\mu} \to F$ is a homomorphism of rings with unity. Let $j_0$ be the smallest integer with $\varphi(E_{j_0}(L)) \neq 0$. Then $j_0$ is the dimension of $F_\varphi \otimes_{\Lambda_{\mu}} \hspace{-0.25 em} M_A(L)$ as a vector space over $F$.
\end{corollary}
\begin{proof}
If $D$ is a diagram of $L$ then Lemma \ref{pres} tells us that $\varphi(M(D))$ is a presentation matrix for the $F$-vector space $F_\varphi \otimes_{\Lambda_{\mu}} \hspace{-0.25 em} M_A(L)$. For a vector space, the only isomorphism-invariant information provided by a presentation matrix is the dimension: an $m \times n$ matrix of rank $r$ is a presentation matrix for a vector space of dimension $n-r$.

The rank $r$ of $\varphi(M(D))$ is the size of the largest square submatrix with nonzero determinant. Determinants are functorial, in the sense that every square $\Lambda_{\mu}$-matrix $X$ has $\varphi(\mathrm{det}X)=\mathrm{det}(\varphi(X))$. It follows that the rank $r$ of $\varphi(M(D))$ is the largest size of a square submatrix $X$ of $M(D)$ with $\varphi(\mathrm{det}X) \neq 0$. If $j_0$ is the smallest index with $\varphi(E_{j_0}(L)) \neq 0$ then the largest size of a square submatrix $X$ of $M(D)$ with $\varphi(\mathrm{det}X) \neq 0$ is $|A(D)|-j_0$, so
\[
\mathrm{dim}_F(F_\varphi \otimes_{\Lambda_{\mu}} \hspace{-0.25 em} M_A(L))=|A(D)|-r=|A(D)|-(|A(D)|-j_0)=j_0.
\]
\end{proof}

Theorem \ref{main2} follows from Lemma \ref{adj} and Corollary \ref{lastcor}.

\section{Two knots}

Inoue~\cite{I} asserted that ``the number of all quandle homomorphisms of a knot quandle to an Alexander quandle is completely determined by Alexander polynomials of the knot.'' Corollary \ref{maincor2} implies a similar assertion, with `Alexander polynomials' replaced by `elementary ideals.' In this section we observe that for Fox tricolorings of the knots pictured in Figure \ref{knotsfig}, Corollary \ref{maincor2} is correct and Inoue's assertion is incorrect.

\begin{figure} [bht]
\centering
\begin{tikzpicture} [scale=0.6]
\draw [thick] (-13+1,0+0.5) -- (-13+1,1+0.5);
\draw [thick] (-13+1,4+0.5) -- (-13+1,3+0.5);
\draw [->] [>=angle 90] [thick] (-13+1,0+0.5) -- (-13+5.5,0+0.5);
\draw [thick] (-13+10,0+0.5) -- (-13+5.5,0+0.5);
\draw [thick] (-13+1,4+0.5) -- (-13+10,4+0.5);
\draw [thick] (-13+1,1+0.5) -- (-13+1.8,1.8+0.5);
\draw [thick] (-13+2.2,2.2+0.5) -- (-13+3,3+0.5);
\draw [thick] (-13+3,3+0.5) -- (-13+5,1+0.5);
\draw [thick] (-13+1,3+0.5) -- (-13+3,1+0.5);
\draw [thick] (-13+3,1+0.5) -- (-13+3.8,1.8+0.5);
\draw [thick] (-13+4.2,2.2+0.5) -- (-13+5,3+0.5);
\draw [thick] (-13+7,1+0.5) -- (-13+5,3+0.5);
\draw [thick] (-13+5,1+0.5) -- (-13+5.8,1.8+0.5);
\draw [thick] (-13+6.2,2.2+0.5) -- (-13+7,3+0.5);
\draw [thick] (-13+9,1+0.5) -- (-13+7,3+0.5);
\draw [thick] (-13+7,1+0.5) -- (-13+7.8,1.8+0.5);
\draw [thick] (-13+8.2,2.2+0.5) -- (-13+9,3+0.5);
\draw [thick] (-13+9,3+0.5) -- (-13+11,3+0.5);
\draw [thick] (-13+9.8,1+0.5) -- (-13+9,1+0.5);
\draw [thick] (-13+10,0+0.5) -- (-13+10,2.8+0.5);
\draw [thick] (-13+10,4+0.5) -- (-13+10,3.2+0.5);
\draw [thick] (-13+11,3+0.5) -- (-13+11,1+0.5);
\draw [thick] (-13+10.2,1+0.5) -- (-13+11,1+0.5);
\node at (-13+1.5,1+0.1) {$u$};
\node at (-13+1.5,3.4+0.5) {$v$};
\node at (-13+3,3.4+0.5) {$w$};
\node at (-13+5,3.4+0.5) {$x$};
\node at (-13+7,3.4+0.5) {$y$};
\node at (-13+10.7,2+0.5) {$z$};

\draw [->] [>=angle 90] [thick] (0,0) -- (4,0);
\draw [thick] (8,0) -- (4,0);
\draw [thick] (0,5) -- (8,5);
\draw [thick] (0,5) -- (0,3);
\draw [thick] (0,0) -- (0,2);
\draw [thick] (2,2) -- (0,2);
\draw [thick] (1,2.2) -- (1,2.8);
\draw [thick] (2,3) -- (0,3);
\draw [thick] (1,1) -- (1,1.8);
\draw [thick] (1,3.2) -- (1,4);
\draw [thick] (4,1) -- (4,4);
\draw [thick] (3,4) -- (4,4);
\draw [thick] (3,1) -- (4,1);
\draw [thick] (1.5,3) -- (0,3);
\draw [thick] (2,3) -- (2.3,3.3);
\draw [thick] (2.7,3.7) -- (3,4);
\draw [thick] (2,2) -- (2.3,1.7);
\draw [thick] (3,1) -- (2.7,1.3);
\draw [thick] (2,1) -- (3,2);
\draw [thick] (2,1) -- (1,1);
\draw [thick] (3,2) -- (3.8,2);
\draw [thick] (5,2) -- (4.2,2);
\draw [thick] (2,4) -- (1,4);
\draw [thick] (2,4) -- (3,3);
\draw [thick] (3.8,3) -- (3,3);
\draw [thick] (4.2,3) -- (5,3);
\draw [thick] (5,3) -- (6,2);
\draw [thick] (6,3) -- (7,2);
\draw [thick] (7,3) -- (8,2);
\draw [thick] (6,2) -- (6.3,2.3);
\draw [thick] (7,3) -- (6.7,2.7);
\draw [thick] (5,2) -- (5.3,2.3);
\draw [thick] (6,3) -- (5.7,2.7);
\draw [thick] (7,2) -- (7.3,2.3);
\draw [thick] (8,3) -- (7.7,2.7);
\draw [thick] (8,3) -- (8,5);
\draw [thick] (8,2) -- (8,0);
\node at (7.6,0.6) {$a$};
\node at (7.6,4.4) {$b$};
\node at (7,1.6) {$c$};
\node at (4.8,3.4) {$d$};
\node at (4.8,1.6) {$e$};
\node at (3.5,4.4) {$g$};
\node at (1.5,4.4) {$h$};
\node at (1.5,0.6) {$i$};
\node at (0.7,2.5) {$j$};
\end{tikzpicture}
\caption{The knots $6_1$ and $9_{46}$.}
\label{knotsfig}
\end{figure}
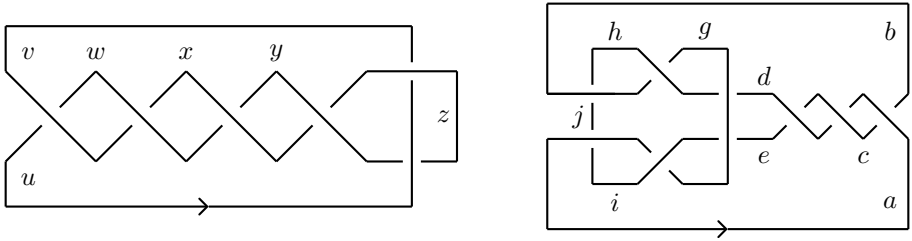

The Alexander polynomials and elementary ideals of the knots $6_1$ and $9_{46}$ were calculated by Crowell and Fox~\cite[Chapter VIII, Examples (4.5) and (4.6)]{CF}. The two knots have the same Alexander polynomials: $\Delta_1=2t^2-5t+2$ and $\Delta_k=1$ for $k>1$. Both knots also have $E_k=(\Delta_k)$ for $k \neq 2$. For $6_1$, $E_2=(\Delta_2)=(1)$ but for $9_{46}$, $E_2=(2-t,1-2t)$. Notice that if $GF(3)$ is the field of three elements then the homomorphism $\varphi:\Lambda \to GF(3)$ with $\varphi(t)=-1$ has
\[
\varphi(2t^2-5t+2)=\varphi(2-t)=\varphi(1-2t)=0.
\]
We see that with respect to this homomorphism $\varphi$, $6_1$ has $j_0=2$ and $9_{46}$ has $j_0=3$.

A Fox tricoloring \cite[Exercise VI.6]{CF} of a link diagram $D$ is a function $f:A(D) \to GF(3)$. At each crossing as in Figure \ref{crossfig}, the sum $f(a_1)+f(a_2)+f(a_3)$ must be $0$ in $GF(3)$. (This is simply the requirement that the coloring satisfies Definition \ref{aquandlecolor}, with $M=GF(3)_\varphi$.) We leave it to the reader to verify the following descriptions of the spaces of Fox tricolorings of $6_1$ and $9_{46}$.

\begin{itemize}
\item Every Fox tricoloring of $6_1$ is given by arbitrary values of $f(u)$ and $f(v)$ in $GF(3)$, with $f(w)=f(z)=-f(u)-f(v)$, $f(x)=f(u)$, and $f(y)=f(v)$.

\item Every Fox tricoloring of $9_{46}$ is given by arbitrary values of $f(a),f(b)$ and $f(g)$ in $GF(3)$, with $f(c)=-f(a)-f(b)$, $f(d)=f(b)$, $f(e)=f(a)$, $f(h)=-f(b)-f(g)$, $f(i)=-f(a)-f(g)$ and $f(j)=f(g)$.

\end{itemize}

It follows that the space of Fox tricolorings of $6_1$ has dimension $2$ over $GF(3)$, and the space of Fox tricolorings of $9_{46}$ has dimension $3$ over $GF(3)$. We see that $6_1$ and $9_{46}$ have different numbers of Fox tricolorings, even though all of their Alexander polynomials are the same.

\section{Two links}

In this section we apply Corollary \ref{maincor2} to the torus link $T_{(2,8)}$ and Whitehead's link $W$, pictured in Figure \ref{linksfig}. 

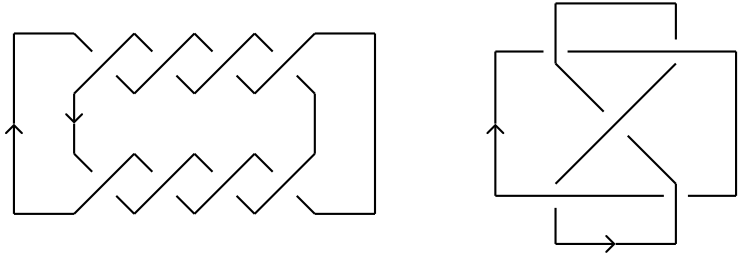
\begin{figure} [bht]
\centering
\begin{tikzpicture}  [scale=0.8]
\draw [->] [>=angle 90] [thick] (0,0+0.5) -- (0,2);
\draw [thick] (1,1+0.5) -- (1,2);
\draw [thick] (1,1+0.5) -- (1.3,0.7+0.5);
\draw [thick] (1.7,0.3+0.5) -- (2,0+0.5);
\draw [thick] (2,1+0.5) -- (2.3,0.7+0.5);
\draw [thick] (2.7,0.3+0.5) -- (3,0+0.5);
\draw [thick] (3,1+0.5) -- (3.3,0.7+0.5);
\draw [thick] (3.7,0.3+0.5) -- (4,0+0.5);
\draw [thick] (4,1+0.5) -- (4.3,0.7+0.5);
\draw [thick] (4.7,0.3+0.5) -- (5,0+0.5);
\draw [thick] (0,0+0.5) -- (1,0+0.5);
\draw [thick] (2,1+0.5) -- (1,0+0.5);
\draw [thick] (3,1+0.5) -- (2,0+0.5);
\draw [thick] (4,1+0.5) -- (3,0+0.5);
\draw [thick] (5,1+0.5) -- (4,0+0.5);
\draw [thick] (5,4-0.5) -- (6,4-0.5);
\draw [thick] (6,0+0.5) -- (6,4-0.5);
\draw [thick] (6,0+0.5) -- (5,0+0.5);
\draw [thick] (0,4-0.5) -- (0,2);
\draw [thick] (0,4-0.5) -- (1,4-0.5);
\draw [->] [>=angle 90] [thick] (1,3-0.5) -- (1,2);
\draw [thick] (1,4-0.5) -- (1.3,3.7-0.5);
\draw [thick] (1.7,3.3-0.5) -- (2,3-0.5);
\draw [thick] (2,4-0.5) -- (2.3,3.7-0.5);
\draw [thick] (2.7,3.3-0.5) -- (3,3-0.5);
\draw [thick] (3,4-0.5) -- (3.3,3.7-0.5);
\draw [thick] (3.7,3.3-0.5) -- (4,3-0.5);
\draw [thick] (4,4-0.5) -- (4.3,3.7-0.5);
\draw [thick] (4.7,3.3-0.5) -- (5,3-0.5);
\draw [thick] (5,1+0.5) -- (5,3-0.5);
\draw [thick] (2,4-0.5) -- (1,3-0.5);
\draw [thick] (3,4-0.5) -- (2,3-0.5);
\draw [thick] (4,4-0.5) -- (3,3-0.5);
\draw [thick] (5,4-0.5) -- (4,3-0.5);
\draw [->] [>=angle 90] [thick] (8,0.8) -- (8,2);
\draw [thick] (8,3.2) -- (8,2);
\draw [thick] (8,3.2) -- (8.8,3.2);
\draw [thick] (9.2,3.2) -- (12,3.2);
\draw [thick] (8,0.8) -- (10.8,0.8);
\draw [thick] (12,0.8) -- (11.2,0.8);
\draw [thick] (12,0.8) -- (12,3.2);
\draw [thick] (9,1) -- (11,3);
\draw [thick] (11,1) -- (10.2,1.8);
\draw [thick] (9,3) -- (9.8,2.2);
\draw [thick] (11,1) -- (11,0);
\draw [->] [>=angle 90] [thick] (9,0) -- (10,0);
\draw [thick] (10,0) -- (11,0);
\draw [thick] (9,0) -- (9,0.6);
\draw [thick] (9,3) -- (9,4);
\draw [thick] (11,4) -- (9,4);
\draw [thick] (11,4) -- (11,3.4);

\end{tikzpicture}
\caption{$T_{(2,8)}$ and Whitehead's link.}
\label{linksfig}
\end{figure}

With the indicated orientations, the elementary ideals of these links are $E_j(T_{(2,8)})=E_j(W)=\Lambda_2$ for $j>1$, $E_j(T_{(2,8)})=E_j(W)=(0)$ for $j<1$, $E_1(W)=(\Delta_1(W))\cdot(t_1-1,t_2-1)=(t_{1}-1)(t_{2}-1)\cdot(t_1-1,t_2-1)$, and
\[
E_1(T_{(2,8)})=(\Delta_1(T_{(2,8)}))\cdot(t_1-1,t_2-1)=(t_{1}^{3}+t_{1}^{2}t_{2}+t_{1}t_{2}^{2}+t_{2}^{3})\cdot(t_1-1,t_2-1).
\]
The elementary ideals may be confirmed using the Alexander matrices obtained from Figure \ref{linksfig}, as described in the introduction. The Alexander polynomials $\Delta_1(T_{(2,8)})$ and $\Delta_1(W)$ may also be verified on the LinkInfo website \cite{Cha}, where the two links are labeled L5a1\{1\} and L8a14\{1\}.

\begin{table} [bht]
\centering
\begin{tabular} {cccc}
$\varphi(t_1)$ & $\varphi(t_2)$ & $j_0(T_{(2,8)})$ & $j_0(W)$ \\[5 pt]
1 & 1 & 2 & 2 \\
1 & -1 & 2 & 2 \\
-1 & -1 & 1 & 1
\end{tabular}
\caption{Values of $j_0$ for homomorphisms $\varphi:\Lambda_2 \to GF(3)$.}
\label{tabl1}
\end{table}

Both links have $E_0=(0)$ and $E_2=\Lambda_2$, so both links have $j_0 \in \{1,2\}$ for every instance of Corollary \ref{maincor2}; if $F$ is a field then a ring homomorphism $\varphi:\Lambda_2 \to F$ yields $j_0=2$ if and only if $\varphi(E_1)=(0)$. Table \ref{tabl1} gives the $j_0$ values for homomorphisms $\varphi:\Lambda_2 \to GF(3)$. (All of the elementary ideals of both links are symmetric with respect to the transposition $t_1 \leftrightarrow t_2$, so we do not need to list the homomorphism with $\varphi(t_1)=-1$ and $\varphi(t_2)=1$; it yields the same $j_0$ values as the homomorphism with $\varphi(t_1)=1$ and $\varphi(t_2)=-1$.) We see that $\textrm{Color}_A(T_{(2,8)},GF(3)_\varphi) \cong \textrm{Color}_A(W,GF(3)_\varphi)$ for every homomorphism of rings with unity $\varphi:\Lambda_2 \to GF(3)$.

The $j_0$ values for homomorphisms $\varphi:\Lambda_2 \to GF(5)$ appear in Table \ref{tabl2}. In the first four rows, we see that $\textrm{Color}_A(T_{(2,8)},GF(5)_\varphi) \cong \textrm{Color}_A(W,GF(5)_\varphi)$ for every homomorphism of rings with unity $\varphi:\Lambda_2 \to GF(5)$ that has $\varphi(t_1)=\varphi(t_2)$. In the last three rows, we see that there are homomorphisms with $\varphi(t_1) \neq \varphi(t_2)$ and $\textrm{Color}_A(T_{(2,8)},GF(5)_\varphi) \ncong \textrm{Color}_A(W,GF(5)_\varphi)$.

\begin{table} [bht]
\centering
\begin{tabular} {cccc}
$\varphi(t_1)$ & $\varphi(t_2)$ & $j_0(T_{(2,8)})$ & $j_0(W)$ \\ [5 pt]
1 & 1 & 2 & 2 \\
2 & 2 & 1 & 1 \\
3 & 3 & 1 & 1 \\
4 & 4 & 1 & 1 \\
1 & 2 & 2 & 2 \\
1 & 3 & 2 & 2 \\
1 & 4 & 2 & 2 \\
2 & 3 & 2 & 1 \\
2 & 4 & 2 & 1 \\
3 & 4 & 2 & 1
\end{tabular}
\caption{Values of $j_0$ for homomorphisms $\varphi:\Lambda_2 \to GF(5)$.}
\label{tabl2}
\end{table}

The links $T_{(2,8)}$ and $W$ are of interest because of the fact that for every abelian group $A$, they have isomorphic groups of Fox colorings in $A$. This fact was verified using Goeritz matrices in \cite[Section 6]{Tcol}, but we can also deduce it from Corollary \ref{maincor3} below, because the homomorphism $\varphi:\Lambda_2 \to \mathbb{Z}$ defined by $\varphi(t_1)=\varphi(t_2)=-1$ has $\varphi(E_j(T_{(2,8)}))=\varphi(E_j(W))$ $\forall j$.

\section{A non-invertible link}

The Laurent polynomial ring $\Lambda_{\mu}=\mathbb{Z}[t_1^{\pm1},\dots,t_{\mu}^{\pm1}]$ has an automorphism given by $t_i \mapsto t_{i}^{-1}$ $\forall i$. This automorphism is sometimes called \emph{conjugation}, and denoted by an overline. Here are two important properties of conjugation.
\begin{enumerate}
\item Let $L^{inv}$ be the \emph{inverse} of an oriented link $L$, obtained by reversing the orientation of every component of $L$. Then $E_j(L)=\overline{E_j(L^{inv})}$ $\forall j$.
\item If $K$ is a knot then $E_j(K)=\overline{E_j(K)}$ $\forall j$.
\end{enumerate}

To verify property 1, let $D$ be a diagram of $L$ and let $D^{inv}$ be the diagram of $L^{inv}$ obtained from $D$ by reversing the orientation of every component. The effect of the orientation reversals is to interchange the indices of the arcs $a_2$ and $a_3$ at every crossing as indicated in Figure \ref{crossfig}. Observe that the effect of (a) interchanging $a_2$ and $a_3$ at every crossing and (b) replacing every $t_i$ with $t_{i}^{-1}$ in the resulting matrix is the same as the effect of (c) multiplying the $a$ column of $M(D)$ by $t_{\kappa(a)}$ for each $a \in A(D)$ and (d) multiplying the $c$ row of $M(D)$ by $-t_{\kappa(a_1)}^{-1}t_{\kappa(a_2)}^{-1}$ for each crossing as indicated in Figure \ref{crossfig}. Property 1 follows because operations (c) and (d) involve multiplying rows and columns by units of $\Lambda_{\mu}$, and hence do not affect the elementary ideals.

Verifying property 2 is more difficult; see \cite[Chapter IX]{CF}.

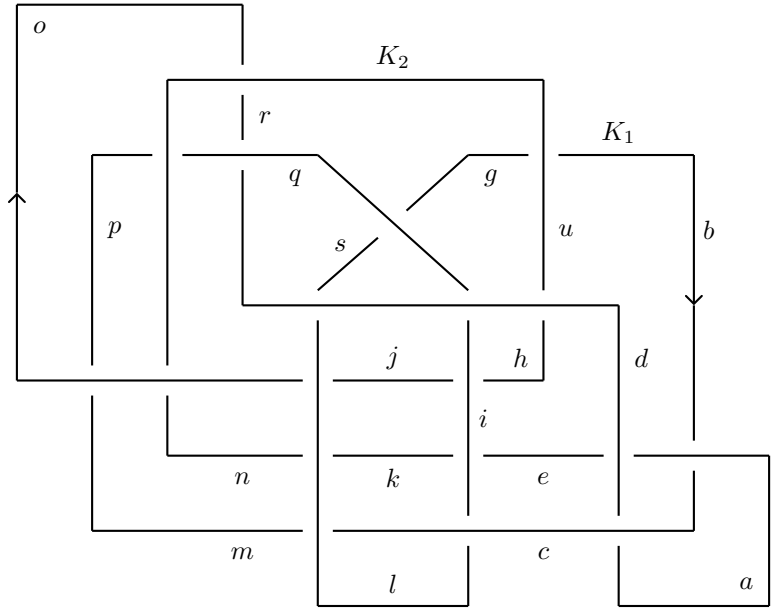
\begin{figure} [bht]
\centering
\begin{tikzpicture} [scale=1.0]
\draw [->] [>=angle 90] [thick] (0,3) -- (0,5.5);
\draw [thick] (0,8) -- (0,5.5);
\draw [thick] (0,8) -- (3,8);
\draw [thick] (3.8,3) -- (0,3);
\draw [thick] (3,7.2) -- (3,8);
\draw [thick] (3,6.2) -- (3,6.8);
\draw [thick] (3,5.8) -- (3,4);
\draw [thick] (8,4) -- (3,4);
\draw [thick] (8,4) -- (8,1.2);
\draw [thick] (8,0) -- (8,0.8);
\draw [thick] (10,0) -- (8,0);
\draw [thick] (10,0) -- (10,2);
\draw [thick] (10,2) -- (8.2,2);
\draw [thick] (7.8,2) -- (6.2,2);
\draw [thick] (5.8,2) -- (4.2,2);
\draw [thick] (3.8,2) -- (2,2);
\draw [thick] (2,2.8) -- (2,2);
\draw [thick] (2,3.2) -- (2,7);
\draw [thick] (7,7) -- (2,7);
\draw [thick] (7,7) -- (7,4.2);
\draw [thick] (7,3.8) -- (7,3);
\draw [thick] (6.2,3) -- (7,3);
\draw [thick] (4.2,3) -- (5.8,3);
\draw [thick] (4,0) -- (4,3.8);
\draw [thick] (4,0) -- (6,0);
\draw [thick] (6,0.8) -- (6,0);
\draw [thick] (6,1.2) -- (6,3.8);
\draw [thick] (1,1) -- (3.8,1);
\draw [thick] (9,1) -- (4.2,1);
\draw [thick] (9,1) -- (9,1.8);
\draw [->] [>=angle 90] [thick] (9,6) -- (9,4);
\draw [thick] (9,4) -- (9,2.2);
\draw [thick] (7.2,6) -- (9,6);
\draw [thick] (6.8,6) -- (6,6);
\draw [thick] (5.18,5.26) -- (6,6);
\draw [thick] (4.8,4.9) -- (4,4.2);
\draw [thick] (2.2,6) -- (4,6);
\draw [thick] (6,4.2) -- (4,6);
\draw [thick] (1,6) -- (1.8,6);
\draw [thick] (1,6) -- (1,3.2);
\draw [thick] (1,2.8) -- (1,1);
\node at (5,7.3) {$K_2$};
\node at (8,6.3) {$K_1$};
\node at (9.7,0.3) {$a$};
\node at (9.2,5) {$b$};
\node at (7,0.7) {$c$};
\node at (8.3,3.3) {$d$};
\node at (7,1.7) {$e$};
\node at (6.3,5.7) {$g$};
\node at (6.7,3.3) {$h$};
\node at (6.2,2.5) {$i$};
\node at (5,3.3) {$j$};
\node at (5,1.7) {$k$};
\node at (5,0.3) {$l$};
\node at (3,0.7) {$m$};
\node at (3,1.7) {$n$};
\node at (0.3,7.7) {$o$};
\node at (1.3,5) {$p$};
\node at (3.7,5.7) {$q$};
\node at (3.3,6.5) {$r$};
\node at (4.3,4.8) {$s$};
\node at (7.3,5) {$u$};
\end{tikzpicture}
\caption{Turaev's non-invertible link, $T$.}
\label{linkfig}
\end{figure}

Properties 1 and 2 indicate that the elementary ideals cannot detect non-invertibility of knots. However the elementary ideals can sometimes detect non-invertibility of links. An example is the two-component link $T$ pictured in Figure \ref{linkfig}, which was discussed by Turaev \cite{T2}. With the indicated component indices and orientations, $T$ has the elementary ideals $E_3(T)=\Lambda_2$ and $E_2(T)=(t_1-3,t_2-1,7)$. (We do not present detailed calculations.) Notice that if $\varphi:\Lambda_2 \to GF(7)$ is the ring homomorphism with $\varphi(t_1)=3$ and $\varphi(t_2)=1$ then $\varphi(E_2(T)) = 0$ but $\varphi(\overline{E_2(T)})$ includes the nonzero element $\varphi(t_{1}^{-1}-3)=5-3=2$. It follows that $E_2(T) \neq E_2(T^{inv})$, so $T$ is not invertible.

Corollary \ref{maincor2} tells us that multivariate Alexander colorings detect the non-invertibility of $T$: the dimension of $\mathrm{Color}_A(D,GF(7)_\varphi)$ over $GF(7)$ is $3$, but the dimension of $\mathrm{Color}_A(D^{inv},$ $GF(7)_\varphi)$ is no more than $2$. We leave it to the reader to verify the following explicit descriptions of these spaces.

\begin{itemize}
\item Every $f \in \mathrm{Color}_A(D,GF(7)_\varphi)$ is given by arbitrary values of $f(a),f(b)$ and $f(i)$ in $GF(7)$, with $f(c)=f(b)-2f(a)$, $f(d)=f(j)=f(k)=-2f(a)$, $f(a)=f(e)=f(h)=f(n)=f(o)=f(r)=f(u)$, $f(g)=2f(a)+f(b)$, $f(l)=4f(a)-2f(b)+3f(i)$, $f(m)=f(i)$, $f(p)=2f(a)+f(i)$, $f(q)=4f(a)+f(i)$ and $f(s)=f(a)-2f(b)+3f(i)$.
\item Every $f \in \mathrm{Color}_A(D^{inv},GF(7)_\varphi)$ is given by arbitrary values of $f(b)$ and $f(i)$ in $GF(7)$, with $f(a)=f(d)=f(e)=f(h)=f(j)=f(k)=f(n)=f(o)=f(r)=f(u)=0$, $f(c)=f(g)=f(b)$, $f(l)=f(s)=3f(b)+5f(i)$, and $f(m)=f(p)=f(q)=f(i)$.
\end{itemize}

\section{Principal ideal domains}

The special theory of modules over principal ideal domains is explained in many algebra books, like \cite{Ja, La, Se}. We summarize the ideas briefly.

Suppose $R$ is a principal ideal domain and $X$ is an $m \times n$ matrix with entries from $R$. Define the elementary ideals $E_j(X)$ as follows: if $j \geq n$, then $E_j(X)=R$; if $n>j\geq \mathrm{max}\{0,n-m\}$, then $E_j(X)$ is the ideal of $R$ generated by the determinants of $(n-j) \times (n-j)$ submatrices of $X$; and if $j<\mathrm{max}\{0,n-m\}$, then $E_j(X)=(0)$. As $R$ is a principal ideal domain, for each integer $j$ there is an $e_j(X) \in R$ such that $E_j(X)$ is the principal ideal generated by $e_j(X)$. Determinants satisfy the Laplace expansion property, so these elements $e_j(X)$ form a sequence of divisors: $e_{j+1}(X) \mid e_j(X)$ $\forall j$. The quotients $d_j(X)=e_j(X)/e_{j+1}(X)$ are the \emph{invariant factors} of $X$. Like the $e_j$, the $d_j$ are well-defined only up to associates, i.e. the principal ideals $(d_j(X))$ are invariants of $X$, but the particular elements $d_j(X)$ are not. The invariant factors also form a sequence of divisors: $d_{j+1}(X) \mid d_j(X)$ $\forall j$. The \emph{Smith normal form} of $X$ is the $m \times n$ matrix obtained from the diagonal matrix

\begin{equation*}%
\begin{pmatrix}
d_{0}(X) & 0 & 0 & 0\\
0 & d_{1}(X) & 0 & 0\\
0 & 0 & \ddots & 0\\
0 & 0 & 0 & d_{n-1}(X)
\end{pmatrix}
\end{equation*}
by adjoining $m-n$ rows of zeroes if $n<m$, and removing $n-m$ rows of zeroes if $n>m$.

The Smith normal form of $X$ is equivalent to $X$, i.e., there are invertible matrices $P,Q$ such that $PXQ$ is equal to the Smith normal form of $X$. It follows that if $X$ is a presentation matrix for the $R$-module $M$, then the Smith normal form of $X$ is also a presentation matrix for $M$. That is, if $X$ is a presentation matrix for $M$ then
\begin{equation}
\label{sumform}
M \cong \bigoplus_{j=0}^{n-1} \text{ } R/(d_j(X)).
\end{equation}
The fact that the $d_j(X)$ form a sequence of divisors implies that $(d_j(X)) \subseteq (d_{j+1}(X))$ $\forall j$. In particular, if $(d_i(X))=R$ then $(d_j(X))=R$ $\forall j\geq i$. Notice that values of $j$ with $(d_j(X))=R$ contribute nothing of significance to the direct sum of (\ref{sumform}).

Using Lemmas \ref{adj} and \ref{pres}, we deduce the following theorem from (\ref{sumform}).


\begin{theorem}
\label{main3}
Let $R$ be a principal ideal domain, and let $\varphi:\Lambda_{\mu} \to R$ be a homomorphism of rings with unity. Suppose $D$ is a diagram of a $\mu$-component link $L$, and $d_0,d_1,\dots,d_{|A(D)|-1}$ are the invariant factors of $\varphi(M(D))$. Then for any $R$-module $M$,
\begin{equation*}
\mathrm{Color}_A(D,M_\varphi) \cong \bigoplus_{j=0}^{|A(D)|-1} \mathrm{Hom}_R(R/(d_j),M)_\varphi.
\end{equation*}
\end{theorem}

The direct sum of Theorem \ref{main3} seems to vary from one diagram to another, but the invariance of the Alexander module guarantees that if $D$ and $D'$ are diagrams of the same link and $|A(D)|<|A(D')|$ then the invariant factors $d'_j$ of $\varphi(M(D'))$ with $j \geq |A(D)|-1$ all generate the same principal ideal, $(d'_j)=R$. It follows that these invariant factors contribute nothing of significance to the direct sum of Theorem \ref{main3}.

\begin{corollary}
\label{maincor3}
Let $R$ be a principal ideal domain, let $\varphi:\Lambda_{\mu} \to R$ be a homomorphism of rings with unity, and let $M$ be an $R$-module. Then for any diagram $D$ of a $\mu$-component link $L$, the $\Lambda_{\mu}$-module $\mathrm{Color}_A(D,M_\varphi)$ is determined up to isomorphism by $M$ and the images under $\varphi$ of the elementary ideals of $L$.
\end{corollary}

The examples of Section 3 show that if we replace ``elementary ideals'' by ``Alexander polynomials'' in Corollary \ref{maincor3} then the resulting statement is false, in general.

Theorem \ref{main3} implies Theorem \ref{main2} in two different ways. (i) Suppose $F$ is the field of quotients of a principal ideal domain $R$. (Perhaps $F=R$.) If $M$ is a vector space over $F$, then $\mathrm{Hom}_R(R/(d_j),M)_\varphi$ is isomorphic to either $(0)$ (if $d_j \neq 0$) or $M_\varphi$ (if $d_j=0$). (ii) Suppose $I$ is a maximal ideal of a principal ideal domain $R$, $F=R/I$ and $M$ is a vector space over $F$. Then $\mathrm{Hom}_R(R/(d_j),M)_\varphi$ is isomorphic to either $(0)$ (if $d_j \notin I$) or $M_\varphi$ (if $d_j \in I$).

We close with thanks to an anonymous reader, who provided helpful comments on the first version of the paper.

\end{document}